\documentclass[11pt,reqno]{amsart}

\usepackage{amssymb}
\usepackage{amscd}
\usepackage{amsfonts}
\usepackage{version}

\numberwithin{equation}{section} \theoremstyle{plain}
\newtheorem{theorem}{Theorem}[section]
\newtheorem{proposition}[theorem]{Proposition}
\newtheorem{lemma}[theorem]{Lemma}
\newtheorem{corollary}[theorem]{Corollary}
\newtheorem{definition}[theorem]{Definition}

\newtheorem*{mainthm3-repeat}{Theorem \ref{mainthm3}}
\newtheorem*{sec3-key-prop-rpt}{Theorem \ref{sec3-key-prop}}

\renewcommand{\leq}{\leqslant}
\renewcommand{\geq}{\geqslant}
\newsavebox{\proofbox}
\savebox{\proofbox}{\begin{picture}(7,7)  \put(0,0){\framebox(7,7){}}\end{picture}}

\newcommand\Z{\mathbb{Z}}
\newcommand\R{\mathbb{R}}

\newcommand\N{\mathbb{N}}

\newcommand\G{\mathbf{G}}
\renewcommand\H{\mathbf{H}}

\newcommand\SL{\operatorname{SL}}

\newcommand\SO{\operatorname{SO}}

\newcommand\GL{\operatorname{GL}}

\newcommand\Mat{\operatorname{Mat}}

\newcommand\rk{\operatorname{rk}}

\newcommand\tr{\operatorname{tr}}
\newcommand\ch{\operatorname{char}}

\newcommand\Sp{\operatorname{Sp}}

\newcommand\F{\mathbb{F}}
\newcommand\Q{\mathbb{Q}}
\newcommand\U{\mathbf{U}}
\newcommand\V{\mathbf{V}}

\newcommand\id{\operatorname{id}}

\def\endproof{\hfill{\usebox{\proofbox}}\vspace{11pt}}

\begin{document}

\title{Strongly dense free subgroups of semisimple algebraic groups}
\author[Breuillard]{Emmanuel Breuillard}
\address{Laboratoire de Math\'ematiques\\
B\^atiment 425, Universit\'e Paris Sud 11\\
91405 Orsay\\
FRANCE}
\email{emmanuel.breuillard@math.u-psud.fr}

\author[Green]{Ben Green}
\address{Centre for Mathematical Sciences\\
Wilberforce Road\\
Cambridge CB3 0WA\\
England}
\email{b.j.green@dpmms.cam.ac.uk}

\author[Guralnick]{Robert Guralnick}
\address{Department of Mathematics \\
University of Southern California\\
3620 Vermont Avenue\\
Los Angeles, California 90089-2532}
\email{guralnick@usc.edu}

\author[Tao]{Terence Tao}
\address{Department of Mathematics, UCLA\\
405 Hilgard Ave\\
Los Angeles CA 90095\\
USA}
\email{tao@math.ucla.edu}


\begin{abstract}
We show that (with one possible exception) there exist \emph{strongly dense} free subgroups in any semisimple algebraic group over a large enough field. These are nonabelian free subgroups all of whose subgroups are either cyclic or Zariski dense. As a consequence, we get new generating results for finite simple groups of Lie type and a strengthening of a theorem of Borel related to the Hausdorff-Banach-Tarski paradox. In a sequel to this paper, we use this result to also establish uniform expansion properties for random Cayley graphs over finite simple groups of Lie type.

\end{abstract}

\maketitle

\setcounter{tocdepth}{1}
\tableofcontents

\section{Introduction and statement of results}\label{sec1}

The existence of free subgroups of algebraic groups is intimately related to the Hausdorff-Banach-Tarski paradox (see \cite{lubotzky-book,wagon} for pleasant introductions to this beautiful subject). To our knowledge, it is Hausdorff\footnote{In fact, Hausdorff showed that if $\psi, \phi$ are rotations of angles $\frac{2\pi}{3}, \pi$ respectively whose axes meet and make an angle $\frac{\theta}{2}$, with $\cos \theta$ transcendental, then $\psi$ and $\phi$ have no relations except $\psi^3=1$ and $\phi^2=1$. They generate the free product $\Z/3\Z * \Z/2\Z$, which of course contains many nonabelian free subgroups, such as the one generated by $a:=\psi\phi\psi$ and $b:=(\phi\psi)^2\phi$.} who in 1914 \cite{hausdorff, hausdorff-book} constructed the first example of a free subgroup of rotations of $\R^3$ in the course of establishing the non-existence of finitely additive rotation-invariant measures on the sphere $S^2$.

That simple algebraic groups $\G(k)$ (say over an infinite finitely generated field $k$) always contain \emph{some} nonabelian free subgroup can nowadays be seen as a trivial consequence of the \emph{Tits alternative} \cite{tits}, according to which every finitely generated subgroup of invertible linear transformations over a field contains a non abelian free subgroup or a solvable subgroup of finite index. However, this fact is not by any means as deep as the Tits alternative and can be proven in a rather simple and constructive way (see Wagon's book \cite{wagon} for several explicit examples).

The problem of finding free subgroups of algebraic groups with special properties has been examined by various authors in the past, often because it was the key to the solution of a seemingly unrelated problem.

In his work on the approximation of continuous Lie groups by discrete subgroups, Kuranishi \cite{kuranishi} proved  in the 1950s that every semisimple real Lie group $G$ admits pairs of elements that generate a dense free subgroup. In fact, he showed that in a certain neighborhood of the identity of $G$, almost every pair of elements has this property (see also \cite{epstein}).

Dekker \cite{dekker1,dekker2} in the 1950s and later Borel \cite{borel-dom} and Deligne-Sullivan \cite{deligne-sullivan} established the existence of free subgroups of rotations acting on the sphere $S^n$ without fixed points (if $n$ is odd) or with commutative isotropy groups (if $n$ is even) thus proving that the sphere is $4$-paradoxical (see \cite{wagon}).

In another direction, T.~Gelander and the first author showed in \cite{breuillard-gelander1, breuillard-gelander2} that every subgroup of a semisimple real Lie group which is dense in the Euclidean topology contains a free subgroup on $2$ generators that is again dense in the Euclidean topology. This fact was then applied in \cite{breuillard-gelander2} to answer an old question of Carri\`ere and Ghys \cite{ghys-carriere} about the rate of volume growth of leaves in Riemannian foliations on compact manifolds.

In this paper, we will be concerned with finding free subgroups with a stronger property than density. Call a subgroup $\Gamma$ of an algebraic group $\G$ \emph{strongly dense} if every pair $x, y \in \Gamma$ of non-commuting elements of $\Gamma$ generate a dense\footnote{Except for Section \ref{secparadox} and unless otherwise stated, all topological notions in this paper will refer to the Zariski topology and we will simply say closed, open or dense, to mean Zariski-closed, Zariski-open or Zariski-dense.}  subgroup of $\G$.  In particular, any subgroup of $\Gamma$ is either abelian or dense.  Our main result is as follows.

\begin{theorem}[Existence of strongly dense free subgroups]\label{sec3-key-prop}
Let $\G(k)$ be a semisimple\footnote{Throughout this paper, semisimple algebraic groups are understood to always be connected.} algebraic group over an uncountable algebraically closed field $k$.   Assume that $\G(k)$ is not $C_2$ with $k$ of characteristic $3$.
Then there exists a free nonabelian subgroup $\Gamma$ of $\G(k)$ on two generators that is strongly dense.
\end{theorem}

We now make some remarks about this theorem.  First of all, the assumption on $k$ can be relaxed in certain ways. For example it is enough to assume that $k$ is of infinite transcendence degree over its prime field. We refer the reader to Section \ref{finite-dense} for precise statements. For countable fields of positive characteristic,  some condition on the transcendence degree is certainly necessary; for instance, all finitely generated subgroups of $\SL_n(\overline{\F_p})$ are contained in some finite group $\SL_n(\F_{p^k})$ and so are obviously not dense.
We do prove weaker results for all algebraically closed fields not algebraic over a finite field;  see Theorem  \ref{char 0}, Corollary
\ref{dense char p} and Theorem \ref{rationality issues}.   It seems very unlikely that the $C_2(k)$ with $k$ of characteristic $3$
is a true exception but it is a true exception to Proposition  \ref{degeneration-prop}, which is a key component of our method.   \vspace{3pt}

Note that once the existence of \emph{some} free subgroup on two generators in a given semisimple algebraic group $\G(k)$ has been established, it is easy to find plenty of such free pairs by a simple Baire category type of argument: for every word $w$ in the free group in two letters, the word variety \[ \mathcal{V}_w:=\{(a,b)\in \G(k) \times \G(k) \,  | \, w(a,b)=\id\}\] is a proper algebraic subvariety of $\G(k) \times \G(k)$. Hence as soon as $k$ is large enough (e.g. if $k$ is uncountable, or is of infinite transcendence degree over its prime field), the union of all word varieties is a proper subset of $\G(k) \times \G(k)$. The same type of argument shows that proving the existence of \emph{some} pair $(a,b)$ that generates a strongly dense free subgroup of $\G(k)$ is equivalent to proving that a ``generic''  pair  of elements of  $\G(k)$ generates a strongly dense free subgroup;  see Theorem \ref{generic strongly dense} for a precise formulation.\vspace{3pt}

While Theorem \ref{sec3-key-prop} is stated for free groups on two generators, it implies the same for free groups with $m$ generators for any $m \geq 2$, since a free group on two generators contains a free group on $m$ generators (and in fact contains a countably generated free subgroup) and any nonabelian subgroup
of a strongly dense subgroup is also strongly dense.

For us, the original motivation for proving Theorem \ref{sec3-key-prop} came while working on expansion properties in finite simple groups of Lie type. A companion paper \cite{bggt1} details these results in full. 




Let us now say a few words about the proof of Theorem \ref{sec3-key-prop}, which spreads over Sections \ref{strongly-dense}
and \ref{sec: degenerations}.  
First, by a similar Baire category type argument as the one we mentioned above, we reduce to proving (see Lemma \ref{singlewords}) a seemingly weaker result. Namely, for every given pair of non-commuting words $w_1,w_2$ in the free group, the set of pairs $(a,b)$ in $\G(k)$ such that $w_1(a,b)$ and $w_2(a,b)$ generate a dense subgroup is generic.

One ingredient in our proof is Borel's theorem \cite{borel-dom} that word maps $w: \G \times \G \rightarrow \G$ associated to every nontrivial word $w \in F_2$ in the free group are all dominant maps. Borel's proof (as well as that of Larsen \cite{larsen-words}) proceeds by  first showing the result for $\G=\SL_n$ by induction on $n \geq 2$, while the general case is proved by induction on $\dim\G$ after observing that every simple algebraic group other than $\SL_n$ contains semisimple subgroups of maximal rank all of whose factors are isogenous to some $\SL_n$.

Our strategy for proving Theorem \ref{sec3-key-prop} and Lemma \ref{singlewords} will be similar. We will first establish the result for $\G=\SL_n$ and then proceed by induction on the dimension of $\G$ via a case-by-case study for each semisimple algebraic group arising from the various root systems. Classical groups are treated in a uniform way, while exceptional groups, especially when the characteristic of the field is $2$, $3$ or $5$, require some additional care.
\vspace{11pt}

We turn now to some corollaries of our main theorem. Proofs of these are given in Sections \ref{finite-dense} and \ref{secparadox}. First, we obtain the following result regarding finite simple groups. 

 \begin{corollary}[Pairs of words generate]\label{finite}  Fix a positive integer $r$. Let $G=\G(\F)$ be a finite \textup{(}simply connected\textup{)} simple group of Lie type of rank
$r$ over the finite field $\F$ \textup{(}$G$ is allowed to be one of the Steinberg twisted groups or a Suzuki-Ree group\textup{)}.
Assume that $G$ is not of the form $\Sp_4(3^e)$ for some $e \geq 1$.   Let $F_2$ be the free group on two generators  and let $w_1,w_2 \in F_2$ be non-commuting words.   Then the probability that $w_1(a,b)$ and $w_2(a,b)$ generate $G$ tends to $1$ as $|\F| \rightarrow \infty$.
\end{corollary}

The proof of Corollary \ref{finite} and most of the results of Section 5 rely crucially on a criterion for when two group elements $x,y$ in $\G$ generate a dense subgroup. If $\ch(k)=0$, this condition is open in $\G \times \G$ (see \cite{gurnewton}). This is not so when $\ch(p)>0$. Instead there is a proper closed subvariety of $\G \times \G$ outside of which a pair $(x,y)$ generates a dense subgroup if and only if it generates an infinite subgroup (see Lemma \ref{gurtiep} and \ref{closed invariants} (i)).\vspace{11pt}

 \noindent \emph{Remark.} If one is interested in the generation properties simply of a random pair $a,b$ then more uniform and more general results are known -- see in particular \cite{kantor-lubotzky} for classical groups and certain exceptional groups and \cite{liebeck-shalev} for the remaining exceptional groups.   Both those papers used the classification of finite
simple groups in the proof (CFSG).   However, the proof in  \cite{kantor-lubotzky} for classical groups of bounded rank is easily modified to avoid this
(using the weaker result in \cite{larsen-pink} that there are only finitely many sporadic simple groups that have linear representations of fixed dimension -- this
does not depend upon CFSG).    On the other hand, the original proof for the larger exceptional groups in \cite{liebeck-shalev} is not so easily modified.   Our methods do not rely on CFSG.
\vspace{11pt}

In another direction, we derive the following consequence of Theorem \ref{sec3-key-prop} in relation with the Hausdorff-Banach-Tarski paradox.

\begin{corollary}[Simultaneous paradoxical decompositions] Let $U$ be a semisimple compact real Lie group. Then there exists $a,b \in U$ such that every homogeneous space of $U$ is $4$-paradoxical with respect to $a$ and $b$. In other words, for every proper closed subgroup $V \leq U$, one can partition $U/V$ into $4$ disjoint sets $A_1 \cup A_2 \cup A_3 \cup A_4$ in such a way that $A_1 \cup aA_2$ and $A_3 \cup bA_4$ are both partitions of $U/V$. In fact almost every pair $a,b$ \textup{(}say with respect to Haar measure\textup{)} in $U$ has these properties. Moreover $F=\langle a, b \rangle $ is free, its action on $U/V$ is locally commutative \textup{(}i.e. point stabilizers are commutative\textup{)} and every non trivial element has $\chi(U/V)$ fixed points, where $\chi(U/V)$ is the Euler characteristic.
\end{corollary}

This result improves Theorem A and Theorem 3 of Borel's paper \cite{borel-dom}, as well as the aforementioned results by Dekker and Deligne-Sullivan, by giving a unified statement. We refer the reader to Section \ref{secparadox} for a discussion of the above corollary and its proof. \vspace{11pt}

We end this introduction by drawing the reader's attention to the following two problems.\vspace{6pt}

\noindent \emph{Problem 1.}
J.~Tits  showed that every finitely generated dense subgroup $\Gamma$ of $\G(k)$ contains a dense free subgroup. This is the beef of the \emph{Tits alternative}. It remains a challenging problem to determine whether or not $\Gamma$ always contains a strongly dense free subgroup.\vspace{6pt}

\noindent \emph{Problem 2.}
 Recall that, according to Borel's theorem \cite[Theorem 1]{borel-dom}, word maps $w: \G \times \G \rightarrow \G$ are dominant, if $\G$ is a semisimple algebraic group and $w$ a non-trivial word in the free group (see also Larsen's paper \cite{larsen-words}). Can one characterize the set of pairs of words $(w_1,w_2)$ in the free group $F_2$ such that the double word map $\G \times \G \rightarrow \G \times \G$ given by $(a,b) \mapsto (w_1(a,b),w_2(a,b))$ is dominant? Clearly a necessary condition is that the cyclic subgroups of $F_2$ generated by $w_1$ and by $w_2$ are not commensurate.   If $w_1$ and $w_2$ are conjugate in $F_2$, then it is quite
 easy to see that the image of the double word map  is contained in the set of conjugate pairs in $\G \times \G$ and so the image
 has dimension at most $2 \dim \G - \rk(\G)$.   The same is true whenever $w_1$ and $w_2$ are contained in a subgroup of $F_2$ generated by two conjugates
(for example,  when $w_2 = [w_1, w_3] := w_1^{-1} w_3^{-1} w_1 w_3$ is the commutator of $w_1$ with another word  --  we see directly that the matrix identity $\tr( A [A,B] ) = \tr(A)$ then provides a non-trivial algebraic constraint on $(w_1(a,b), w_2(a,b))$).   See \cite{horowitz} for some interesting examples when $\G=\SL_2$.   \vspace{6pt}





\textsc{Acknowledgments.} EB is supported in part by the ERC starting grant 208091-GADA.
BG was, for some of the period during which this work was being carried out, a Fellow at the Radcliffe Institute at Harvard. He is very happy to thank the Institute for providing excellent working conditions.
RG is supported by NSF grant DMS-1001962.   He also thanks the Institute for Advanced Study, Princeton.
TT is supported by a grant from the MacArthur Foundation, by NSF grant DMS-0649473, and by the NSF Waterman award.

\section{Strongly dense free subgroups of $\G(k)$}\label{strongly-dense}

We turn now to the proof of our main theorem. Let us begin by recalling the statement, namely that if $\G(k)$ is a semisimple algebraic group over an uncountable algebraically closed field $k$, then there is a non-abelian free subgroup $\Gamma \leq \G(k)$ that is strongly dense. This means that every pair of noncommuting elements of $\Gamma$ generates a dense subgroup of $\G(k)$. \vspace{11pt}

We will switch between various forms of our semisimple group -- either the adjoint form or simply connected form or something in between, depending upon
which is the easiest form to work with.  Since the center is contained in the Frattini subgroup, this has no effect on generation results or on whether
subgroups are free.
We begin by outlining the main steps of the argument. The rest of the section will be devoted to proofs of these steps, except for a certain key ingredient, Proposition \ref{degeneration-prop} below, which involves a rather long case check. We defer this to Section \ref{sec: degenerations}.

Here, then, are the key steps in the proof of Theorem \ref{sec3-key-prop}. As remarked in the introduction, the overall strategy is to proceed by induction on $\dim(\G)$. Assume, then, that the result is known for all semisimple $\G'(k)$ with $\dim \G' < \dim \G$ (aside those semisimple groups which involve
$C_2$ in characteristic $3$ as a factor).

\vspace{11pt}

\emph{Step 1:} reduction to looking at single pairs of words. This was also mentioned in the introduction; we add a little more detail now. Using a fairly standard ``category'' (in the sense of Baire) type of argument we reduce matters to the task of proving the following lemma. Hereinafter, a \emph{generic} set is one that contains the complement of a countable union of proper closed subvarieties. The complement of a generic set (that is, a set contained in a countable union of proper closed subvarieties) is called \emph{meagre}.

\begin{lemma}\label{singlewords} Suppose that $\G(k)$ is a semisimple algebraic group over an uncountable algebraically closed field $k$.
Assume that if the characteristic of $k$ is $3$, then $\G(k)$ has no factors of type $C_2$.
Let $w, w' \in F_2$ be non-commuting words.  Then for a generic set of $(a,b) \in \G(k) \times \G(k)$, $w(a,b)$ and $w'(a,b)$ generate a dense subgroup of $\G(k)$.
\end{lemma}

In fact, if $\ch(k)=0$, we can replace generic by open and remove the uncountable hypothesis (see Theorem \ref{char 0} for details).
Since there are only countably many pairs of words to consider, we immediately obtain the following corollary:

\begin{theorem} \label{generic strongly dense}  Suppose that $\G(k)$ is a semisimple algebraic group over an uncountable algebraically closed field $k$.
If $\ch(k)=3$, assume that $\G(k)$ has no factors of type $C_2$.
 Then for a generic set of $(a,b) \in \G(k) \times \G(k)$, $\langle a, b \rangle$ is a strongly dense free group.
\end{theorem}

In view of Lemma \ref{uncountable} below, we see that Theorem \ref{generic strongly dense} implies Theorem \ref{sec3-key-prop}.

\emph{Step 2:} reduction to the simple case. In this step, we show that Lemma \ref{singlewords} for semisimple $\G$ follows from the special case in which $\G$ is simple.
\vspace{11pt}

\emph{Step 3:} examination of the closure $\overline{\langle w(a,b), w'(a,b)\rangle}$ of $\langle w(a,b), w'(a,b)\rangle$, where $w, w' \in F_2$ are noncommuting words. This is a key step towards the proof of Lemma \ref{singlewords}. In this step, making use of Borel's result on dominance of word maps (Proposition \ref{borel-words-prop}), we prove a weaker variant of Theorem \ref{generic strongly dense}, namely that for a generic set of $(a,b) \in \G(k) \times \G(k)$, $w(a,b)$ and $w'(a,b)$ generate a group which is either (i) dense in $\G(k)$ or (ii) contained in a maximal, connected, semisimple proper subgroup $\H < \G(k)$ with $\rk(\H) = \rk(\G)$. We note that if $\G$ is a special linear group $\SL_n$, or a direct product of such groups, then we are already done at this point. This is because $\SL_n$ has no proper maximal connected semisimple subgroups of maximal rank. \vspace{11pt}

\emph{Step 4:} understanding degenerations of subgroups. To complete the proof of  Lemma \ref{singlewords}, we must rule out option (ii) in Step 3, or at least show that option (i) occurs generically. This is by far the hardest part of the argument, but the basic strategy is not too hard to describe. Suppose for contradiction that option (i) of Step 3 does not occur generically. By another fairly simple ``Baire category'' argument and Borel's result
about the dominance of the word map, there must be a single connected, semisimple proper subgroup $\H < \G(k)$ of the same rank as $\G$ and a pair of noncommuting words $w,w' \in F_2$ such that, for a generic set of $(a,b) \in \G(k) \times \G(k)$, $w(a,b)$ and $w'(a,b)$ lie together in a conjugate $\H^g := g\H g^{-1}$ of $\H$ ($g$ may depend on $a$ and $b$). In particular the pair $(w(a,b), w'(a,b))$ lies in the closure
\[ \overline{\bigcup_{g \in G} (\H^g \times \H^g)}\] for \emph{every} pair $a,b \in \G(k)$.

Now we make an important definition.  Call a proper closed semisimple subgroup $\H'$ a \emph{degeneration} of $\H$ in $\G$ if there exists a pair
\[ (u,v) \in \overline{\bigcup_{g \in G} (\H^g \times \H^g)}\] such that the closure of the group generated by $u$ and $v$ is $\H'$. It will follow from the induction-on-dimension hypothesis that given any proper closed semisimple subgroup $\H'$, there  exist $a,b \in \H'$ such that $u := w(a,b)$ and $v := w'(a,b)$ generate a dense subgroup of $\H'$. In particular, every such subgroup $\H'$ will be a degeneration of $\H$.  On the other hand, in Section \ref{sec: degenerations}, we will show (by a somewhat tedious case check) that for every connected semisimple proper subgroup $\H < \G(k)$, there exists a proper closed semisimple subgroup $\H'$ which is \emph{not} a degeneration of $\H$.  This contradiction completes the proof of Theorem \ref{sec3-key-prop}.

For sake of reference, we now state the main result from Section \ref{sec: degenerations} mentioned above formally.

\begin{proposition}\label{degeneration-prop}
Suppose that $\G(k)$ is a simple algebraic group over an uncountable algebraically closed field $k$. 
 Let $\H < \G(k)$ be a proper connected semisimple maximal subgroup
of $\G$ with $\rk(\H)=\rk(\G)$. 
Assume that if $\ch(k)=3$, $\G(k)$ is not $C_2$.   
 Then there is a proper closed semisimple subgroup $\H' < \G(k)$ such that $\H'$ has no factor of type $C_2$, and with the property that if
\[ (u,v) \in \overline{\bigcup_{g \in G} (\H^g \times \H^g)} \] then the closure of the group generated by $u$ and $v$ is  not $\H'$.
\end{proposition}

Note that $C_2(k)$ with $\ch(k)=3$ is a true exception to the result above.   The only semisimple subgroups of $C_2(k)$
in characteristic $3$ are a conjugacy class of $A_1 \times A_1$ and two conjugacy classes of $A_1$.  However, the latter two subgroups can be shown to be conjugate to a subgroup of the $A_1 \times A_1$ subgroup.

We now turn to the task of providing further details for each of the four steps above, starting with \emph{Step 1}, the deduction of Theorem \ref{sec3-key-prop} from Lemma \ref{singlewords}. Since there are only countably many words in $F_2$, and a countable intersection of generic sets is manifestly generic, Lemma \ref{singlewords} immediately implies that for generic $(a,b) \in \G(k) \times \G(k)$, $w(a,b)$ and $w'(a,b)$ generate a dense subgroup of $\G(k)$ for all noncommuting pairs of words $w, w'$. To deduce Theorem \ref{sec3-key-prop} from this, it suffices to show that a generic set is nonempty. This is a standard observation, but for completeness we include a proof.

\begin{lemma} \label{uncountable}
 Let $V$ be an irreducible algebraic variety over an uncountable algebraically closed field $k$.  Then every generic subset of $V$ is non-empty.
\end{lemma}
\begin{proof}
We may assume without loss of generality that $V \subset k^n$ is an irreducible affine variety.  The case $n \leq 0$ is trivial, so assume inductively that $n \geq 1$ and that the claim has been proven for smaller values of $n$.  We may then reduce to the case when $V$ is not contained in any hyperplane.  In particular, for all but finitely many $t$, the slice $V \cap (k^{n-1} \times \{t\})$ is the finite union of irreducible varieties of dimension exactly one less than that of $V$.  Assume for contradiction that $V$ is the countable union of proper subvarieties $W_1, W_2, \ldots$, which we may of course assume to be non-empty.  We then split $k^n$ as $k^{n-1} \times k$ and note that for each $j$, the slices $W_j \cap (k^{n-1} \times \{t\})$ are either empty, or have dimension strictly less than that of $W_j$ for all but finitely many $t$.  As $k$ is uncountable, we may thus find a $t$ for which all the above conclusions concerning the dimension of $V \cap (k^{n-1} \times \{t\})$ and the $W_j \cap (k^{n-1} \times \{t\})$ hold. The claim then follows from the induction hypothesis.
\end{proof}

\noindent \emph{Remark.} The above lemma continues to hold if $k$ is only assumed to have infinite transcendence degree over its prime field and generic sets are assumed to contain the intersection of at most countably many open subvarieties all defined over a given finitely generated subfield of $k$. See \cite[Lemma 2]{borel-dom}.\vspace{6pt}

We move on now to \emph{Step 2}, the deduction of the semisimple case of Lemma \ref{singlewords} from the simple case. By passing to the universal cover we may assume without loss of generality that $\G = \G_1 \times \dots \times \G_t$ is a product of simple algebraic groups. We will require two ingredients.

\begin{proposition}[Borel]\label{borel-words-prop}
 Let $\G(k)$ be a semisimple algebraic group  over an algebraically closed field.  Then any nontrivial word $w$ in the free group $F_2$ gives rise to a dominant map $\G(k) \times \G(k) \rightarrow \G(k)$.  In particular, if $\Omega$ is a subset of $\G(k)$, $\Omega$ is a generic subset of $\G(k)$ if and only if $\{ (a,b) \in \G(k) \times \G(k): w(a,b) \in \Omega \}$ is a generic subset of $\G(k) \times \G(k)$.
\end{proposition}

\begin{proof} See Borel \cite{borel-dom} or Larsen \cite{larsen-words}.
\end{proof}

\noindent \emph{Remark.} As we noted earlier, the joint map $(a,b) \mapsto (w(a,b), w'(a,b))$ need not be dominant even if $w$ and $w'$ do not commute.  This defeats a naive attack on Lemma \ref{singlewords}, and forced us to use the argument below based on degenerations
 instead.  However, a complete solution to Problem 2 in the introduction would likely lead, among other things, to a new proof of Lemma \ref{singlewords} and hence Theorem \ref{sec3-key-prop}, which in particular could resolve the case of $C_2$ in characteristic $3$.

An important consequence of Proposition \ref{borel-words-prop} is the following.

\begin{lemma} \label{generic-max-rank}  Let $\G(k)$ be a semisimple algebraic group  over an uncountable algebraically closed field.
If $w$ is a nontrivial word in the free group $F_2$, then the set of pairs $(a,b) \in \G(k) \times \G(k)$ such that $\overline{\langle w(a,b) \rangle}$
is a maximal torus of $\G(k)$ is a generic subset of $\G(k) \times \G(k)$.
\end{lemma}

\begin{proof}   Fix a maximal torus $T$ of $\G(k)$.  Note that there are only countably many proper subtori $T_i$ of $T$.
For each such torus $T_i$, write $W_i := \bigcup_{g \in \G(k)} T_i^g$.   Then $\overline{W_i}$ is a proper closed subvariety of $\G(k)$, whence
$\bigcup_i  \overline{W_i}$ is a countable union of proper subvarieties and is meagre.  The claim follows.
\end{proof}

\newcommand\Inn{\operatorname{Inn}}
\newcommand\Aut{\operatorname{Aut}}

We return to the discussion of \emph{Step 2} (the reduction to the simple case). Fix noncommuting words $w,w' \in F_2$. Let $\G= \G_1\times \ldots \times \G_t$
where each $\G_i$ is simple.  Let us examine the pairs $a = (a_1,\dots,a_t)$, $b = (b_1,\dots,b_t) \in \G \times \G$. If we assume that Lemma \ref{singlewords} has already been established for simple groups, then for each $i=1,\ldots,t$, it will generically be the case that $w(a_i, b_i)$ and $w'(a_i,b_i)$ generate a dense subgroup of $\G_i$.   By Lemma \ref{generic-max-rank}, generically the closure of $\langle w(a,b) \rangle$ is a maximal torus of $\G$.  Let $H$ be the closure
of $\langle w(a,b), w'(a,b) \rangle$.   Since $H$ contains a maximal torus $T$ of $G$, it contains a maximal torus $T_i$ of $\G_i$.   Since $H$ projects onto
each $\G_i$, it follows that $H$ contains all $\G_i$-conjugates of $T_i$, and hence all semisimple elements in $\G_i$, whence $\G_i \le H$.  Thus, $H = \G$ as required.


We move now to the discussion of \emph{Step 3} of the outline, in which a weaker variant of the key Lemma
 \ref{singlewords} is established. Here is a formal statement of what we prove.

\begin{lemma}[Step 3]\label{step3-lem}
Suppose that $\G(k)$ is a semisimple algebraic group and that $w,w' \in F_2$ are noncommuting words.
Then for generic $(a,b) \in \G(k) \times \G(k)$ the elements $w(a,b), w'(a,b)$ generate a group whose
closure $\overline{\langle w(a,b), w'(a,b)\rangle}$ is infinite and is either \textup{(i)}  $\G(k)$ or \textup{(ii)} a connected proper semisimple subgroup
$\H < \G(k)$ with $\rk(\H) = \rk(\G)$.
\end{lemma}
\begin{proof}
By Lemma \ref{generic-max-rank}, it is clear that the closure of $\langle w(a,b), w'(a,b)\rangle$, which contains the closure of $\langle w(a,b)\rangle$,
generically contains a maximal torus and hence is infinite. If, for such a pair $(a,b)$, this group is not all of $\G(k)$,
it is therefore contained in a maximal closed proper subgroup $H < \G(k)$. By the observation above,
it is generically the case that all three of $w(a,b), w'(a,b)$ and $[w(a,b),w'(a,b)]$ generate subgroups whose closures are maximal tori.
It suffices to show that, in this case, $\H$ is connected, semisimple and has $\rk(\H) = \rk(\G)$.

The statement about ranks is immediate, since $\H$ contains a maximal torus in $\G(k)$.
$\H$ is generated by two maximal tori, which in particular are connected, whence $\H$ is connected.

By a theorem of Borel and Tits (see \cite[3.1.3]{gls3} or \cite[10.4]{humphreys}), any closed subgroup $\H < \G(k)$ is either reductive or is contained in a maximal parabolic subgroup. In the former case the fact that $\rk ([\H,\H]) = \rk(\G)$ implies that $\H$ has no central torus, and hence is semisimple and we are done. In the latter case, suppose that $\H$ is contained in a maximal parabolic subgroup $P$ with Levi decomposition $P = LU$, where $L$ is a Levi subgroup and $U$ is the unipotent radical.  The Levi subgroup is generated
by a maximal torus and a collection of root subgroups  corresponding to a proper subset
of the simple roots.     Thus, the semisimple rank of $L$ is the cardinality of this set of roots
and in particular is less than the rank of $\G$.
\end{proof}

Finally, we turn to the analysis of \emph{Step 4} of the outline, which consists in bootstrapping Lemma \ref{step3-lem} (for simple $\G$) to the stronger statement of Lemma \ref{singlewords}.

It suffices to show that the set of pairs $(a,b) \in \G(k) \times \G(k)$ for which $w(a,b), w'(a,b)$ generate a group whose closure is some connected, proper semisimple $\H < \G(k)$ with $\rk(\H) = \rk(\G)$ is meagre. To study this further, we will use the following facts.

\begin{enumerate}
\item There are only finitely many conjugacy classes of maximal, connected, proper, semisimple $\H < \G(k)$ with $\rk(\H) = \rk(\G)$;
\item For each such $\H$ the set
\[ \bigcup_{g \in \G(k)} (\H^g \times \H^g)\] is \emph{constructible}, that is to say is a finite boolean combination of open and closed sets in $\G(k)$.
\end{enumerate}

Fact (i) is well-known and we refer the reader to the table of Section \ref{sec: degenerations} below for a complete list of maximal semisimple subgroups of maximal rank up to conjugacy and to \cite{borel-de-siebenthal,liebeck,liebeck-seitz} for a discussion and proof of this fact.
Fact (ii) is a trivial consequence of the well-known property of the class of constructible sets that it is stable under algebraic morphisms (see \cite{borel-alggroups} for a proof of this property).

Since a finite union of meagre sets is meagre, fact (i) implies that we may focus attention on a single conjugacy class of $\H$. Thus it suffices to show that, for each \emph{fixed} $\H$, the set $X$ of pairs $(a,b) \in \G(k) \times \G(k)$ for which
\[ (w(a,b), w'(a,b)) \in \bigcup_{g \in \G(k)} (\H^g \times \H^g)\] is meagre.
The set $X$, being the preimage of a constructible set under an algebraic map, is itself constructible. But every constructible set is either meager or generic. Indeed, write $1_X$ in a minimal fashion as a finite $\pm 1$ combination of characteristic functions $1_S$ of closed sets. If $S = \G(k)\times \G(k)$ is not one of these sets then $X$ is meagre. Otherwise, $X$ contains the complement of finitely many closed subsets of $\G(k)\times \G(k)$ and in particular is generic and thus dense. It follows that for \emph{every} pair $(a,b) \in \G(k) \times \G(k)$ we have
\begin{equation}\label{to-contradict} (w(a,b), w'(a,b)) \in \overline{\bigcup_{g \in \G(k)} (\H^g \times \H^g)}\end{equation}

We then invoke Proposition \ref{degeneration-prop} to obtain an auxiliary proper, closed, semisimple subgroup $\H' < \G(k)$ with the property that there does not exist
\[ (u,v) \in \overline{\bigcup_{g \in \G} (\H^g \times \H^g)}\] such that $\overline{\langle u, v\rangle} = \H'$. As already mentioned, the existence of such an $\H'$ is established in Section \ref{sec: degenerations}. By the main induction hypothesis (in the proof of Theorem \ref{sec3-key-prop}) there are $a,b \in \H' \subseteq \G(k)$ such that $\overline{\langle  w(a,b), w'(a,b)  \rangle} = \H'$. This contradicts \eqref{to-contradict} and the choice of $\H'$.

This concludes the proof of Theorem \ref{sec3-key-prop}, except of course for the construction of $\H'$.

\vspace{11pt}

\section{Degenerations of maximal rank semisimple subgroups} \label{sec: degenerations}

The purpose of this section is to prove  Proposition \ref{degeneration-prop}, and thus complete the proof of Theorem \ref{sec3-key-prop}. In the process of doing so, we will also get some information on degenerations of maximal rank subgroups of classical simple algebraic groups which may be of independent interest. The definition of \emph{degeneration} was given just before the statement of Proposition \ref{degeneration-prop}; we formalise it now.

Let $\G$ a simple algebraic group over an algebraically closed field $k$ of rank $r$.   We assume
that $k$ is not algebraic over a finite field.
If $\H$ is a closed subgroup of $\G$,  let \[ W_1(\H) := \bigcup_{g \in \G} \H^g\] and
\[ W_2(\H):= \bigcup_{g \in \G} (\H^g \times \H^g).\]

\begin{definition}[Degenerations] \label{def:degen}  Let $\H$ and $\H'$ be a closed subgroups of  $\G$.
We say that $\H'$ is a \emph{degeneration} of $\H$ \textup{(}with respect to $\G$\textup{)} if there exist $(u,v)  \in \overline{W_2(\H)}$ with
$\langle u, v \rangle$ dense in $\H'$.
\end{definition}

Our aim in this section is to prove Proposition \ref{degeneration-prop}. This is the assertion that, given any proper connected semisimple subgroup $\H \leq \G$ with $\rk(\H) = \rk(\G)$,  there exists another proper semisimple subgroup $\H'$ that is not a degeneration of $\H$.

We make some remarks on $W_1(\H)$ and $W_2(\H)$. First of all we note that in general,  $W_1(\H)$ is not a closed subvariety of $\G$ and hence $W_2(\H)$ certainly need not be a closed subvariety of $\G \times \G$.
For example, if $\H$ is a subgroup of maximal rank,  then $W_1(\H)$ contains all semisimple elements and therefore its closure
is $\G$ (and rarely is every element of $\G$ conjugate to an element of a proper subgroup).   Similarly, if $\G= \SL_{2n}$ and $\H=\Sp_{2n}$,
then $\H$ contains a regular unipotent element  of $\G$ and so the closure of $W_1(\H)$ contains all unipotent elements of $\G$ (and again,
$\H$ will not contain conjugates of all unipotent elements).

There is, however, one case in which $W_2(\H)$ \emph{is} closed.

\begin{lemma} \label{parabolic}   If $\H$ is a parabolic subgroup of $\G$, then $W_2(\H)$ is closed.
\end{lemma}

\begin{proof}  Let $X=\G/\H$.  By definition \cite[11.2]{borel-alggroups} since $\H$ is a parabolic subgroup,  $X$ is a complete variety.
Let $f:\G \times \G \times X \rightarrow \G \times \G$ be the natural projection.   Since $X$ is complete,
$f$ is a closed map \cite[GM 7.4]{borel-alggroups}.    Let \[ Y=\{(a,b,x) \in \G \times \G \times X \, | \, ax=bx=x\}.\]  Then $f(Y)$ is closed.
Since $f(Y)=W_2(\H)$, the result follows.
\end{proof}

We will give some necessary conditions for $\H'$ to be a degeneration of $\H$. Then we will use these conditions
to verify Proposition  \ref{degeneration-prop}.

First we note the following lemma.

\begin{lemma} \label{W1}  Suppose that $\H$ is a connected reductive subgroup of rank $s$.   Then $\dim \overline{W_1(\H)}  \le \dim \G - (r-s)$.
\end{lemma}

\begin{proof}   Since $\H$ is reductive, the set of semisimple elements of $\H$ contains an open dense subset of $\H$.
Thus, if we let $S$ be a maximal torus of $\H$, $\bigcup_{h \in \H} S^h$ is dense in $\H$ and so
$\overline{W_1(\H)} = \overline{W_1(S)}$.   Consider the morphism $f: \G \times S \rightarrow W_1(\H)$ given by conjugation: $f(g,s) := gsg^{-1}$.  We have seen
that this is a dominant map.  Clearly every fiber has dimension at least $r$ (since we can embed $S$
in a maximal torus of $\G$), and so the result follows.
\end{proof}

We say that a finite-dimensional $k$-vector space $V$ is a \emph{$\G$-module} if it is equipped with a linear rational action of $\G$, that is to say an algebraic homomorphism $\G \rightarrow \GL(V)$.
We now produce some closed varieties associated with such modules which will provide some necessary conditions on degenerations.

\begin{lemma}[Closed invariants]\label{closed invariants}  Suppose that $\G$ is a simple algebraic group over an algebraically closed field $k$, acting on a $\G$-module $V$.  Let $d, M $ be   positive integers.
 \begin{enumerate}
\item    The set of pairs $(a,b) \in \G \times \G$ that have a common $d$-dimensional invariant subspace in $V$ is closed.
\item    The set of pairs $(a,b) \in \G \times \G$ such that
$$
\sum_{ d_i(a,b)  > d}  d_i(a,b)   \leq M
$$
is closed where the $d_i(a,b)$ are the dimensions of the composition factors of $\langle a, b \rangle$ on $V$.
\end{enumerate}
\end{lemma}

\begin{proof}  Item (i) follows just as in the proof of Lemma \ref{parabolic}, using the fact the Grassmanian is a complete variety.

We now prove (ii).   Let $F(A,B)$ be a polynomial such that $F(A,B)$ does not vanish on
the algebra $\Mat_{d+1}(k)$ of $(d+1) \times (d +1)$ matrices over $k$ but does vanish on $\Mat_d(k)$.   See \cite[\S 1.3]{rowen}
for a discussion of these polynomials.
The existence of such a polynomial  is essentially the  Amitsur-Levitzky theorem (by modifying the standard identity one can get a two variable
identity).       Let $(a,b) \in \G \times \G$.   Choose a composition
series
\[ \{0\}=V_0 \subset V_1 \subset \dots \subset V_m = V\] for $\langle a, b \rangle$ with $d_i = \dim (V_{i+1}/V_i)$.
If $d_i \leq d$, it follows that $F(a', b')= 0$ on $V_{i+1}/V_i$ for every $a', b'$ in the algebra $\mathcal{A}$  generated by $a,b$ (in $\mathrm{End}(V)$).
  If $d_i > d$, it is easy
to see that $F(a', b')$ is generically  invertible on $V_{i+1}/V_i$ for $a', b' \in \mathcal{A}$.  It follows that for $N$ sufficiently large ($N \geq \dim V$
certainly suffices) the rank of $F(a',b')^N$ is generically $\sum_{d_i > d}  d_i$.   Therefore the set described in (ii) is precisely the set of
$(a,b)$ such that $F(a',b')^N$ has rank at most $M$ for all $a', b'$  in the algebra generated by $a,b$, and this is clearly a closed condition.
\end{proof}

This lemma has the following immediate consequence.

\begin{corollary}[Necessary conditions for degeneration]\label{test} Let $\H, \H'$ be   closed subgroups of $\G$.   Let $V$ be a $\G$-module.   Let $d_i, d_i'$ denote
the dimensions of the composition factors of $\H$ and $\H'$, respectively, on $V$.    Suppose that
$\H'$ is a degeneration of $\H$.   Let $d$ be a positive integer.
\begin{enumerate}
\item  If $\H$ has a $d$-dimensional invariant subspace, then so does $\H'$.  In particular, if $\H'$ acts irreducibly on $V$, then so must $\H$.
\item  $ \sum_{1 \leq i \leq r: d_i > d} d_i \geq \sum_{1 \leq i \leq r: d'_i > d} d'_i.$
\end{enumerate}
\end{corollary}

We now begin the proof of Proposition \ref{degeneration-prop} for classical groups (in fact we prove a stronger result in that case). After that we go through the exceptional groups one by one.  We use the well-known classification of maximal rank semisimple subgroups for each such group, inspecting their behaviour on a suitable finite dimensional module, and then applying Corollary  \ref{test}.  Broadly speaking, when the group $\G$ is classical we will rely on part (i) of Corollary \ref{test}, whilst for the exceptional algebraic groups we will rely mostly on part (ii).

The connected semisimple maximal rank subgroups of simple algebraic groups are well understood.
In characteristic $0$, these are described in terms of the extended Dynkin diagram (this was originally proved by
Borel and de Siebenthal \cite{borel-de-siebenthal} in the setting of compact groups and is valid in almost all characteristics).
If the characteristic $p > 0$ is positive but small, there are some extra cases; and
for exceptional groups, one can compute the various possibilities by considering subsystems of roots.
We will use the paper \cite[Tables 5.1, 5.2]{lss} as a reference for the maximal rank subgroups in the exceptional groups.
We also refer the reader to \cite{liebeck,liebeck-seitz,lieseitz,seitz-mem} for a complete description of
maximal rank subgroups of simple algebraic groups. For the reader's convenience, we summarize them (up to conjugacy and isogeny)
in the following table.
Note that $D_2 = A_1 \times A_1$, that $B_2=C_2$ when the characteristic is not $2$.  \vspace{6pt}

\noindent\begin{tabular}{|l|l|}
\hline
Simple group of rank $r$ & Maximal semisimple subgroups of rank $r$ \\
\hline
$A_{r},r\geq 1$ & none \\
$B_{r},r\geq 2, p \ne 2 $ & $D_{r},B_{k}\times D_{r-k}$ for $1\leq k  <  r-1$ \\
$C_r, r \geq 2, p =2$  &  $D_r, C_k \times C_{r-k}$ for $1 \leq k \leq \lfloor r/2 \rfloor$ \\
$C_{r},r\geq 3, p \ne 2$ & $C_{r-k}\times C_{k}$ for $1\leq k\leq  \lfloor r/2 \rfloor$ \\
$D_{r},r\geq 4$ & $D_{k}\times D_{r-k}$ for $2\leq k\leq \lfloor r/2 \rfloor$ \\
$E_{6}$ & $A_{1}\times A_{5},A_{2}\times A_{2} \times A_2$ \\
$E_{7}$ & $A_{1}\times D_{6},A_{7},A_{2}\times A_{5}$ \\
$E_{8}$ & $D_{8},A_{1}\times E_{7},A_{8},A_{2}\times E_{6},A_{4}\times A_{4}$ \\
$F_{4}$ & $A_{1}\times C_{3},B_{4},A_{2}\times A_{2}$ (and if $p=2$ another $B_{4}$) \\
$G_{2}$ & $A_{1}\times A_{1},A_{2}$ (and if $p=3$ another $A_{2}$) \\
\hline
\end{tabular}
\vspace{6pt}

For the classical groups $\G = A_n, B_n, C_n, D_n$, the following partial classification of maximal rank semisimple subgroups will be useful.
If $\G$ is a classical group, we let $V$ denote its natural module (i.e the module of dimension $n+1, 2n+1, 2n, 2n$ respectively).

\begin{lemma}\label{max-rank}  Let $\G$ be a simple classical group over an algebraically closed field $k$
of characteristic $p \geq 0$ with natural module $V$.
If $\H$ is a proper closed connected subgroup of $\G$ of maximal rank, then one of the following holds:
\begin{enumerate}
\item The action of $\H$ on $V$ is reducible, i.e. $\H$ stabilises a proper subspace of $V$;
\item  $\G=\Sp_4$ and $\H=\SO_4$;  or
\item  $p=2$,   $\G=\Sp_{2n}$ and $\H=\SO_{2n}$.
\end{enumerate}
\end{lemma}

\begin{proof}  If $\H$ has a unipotent radical $U$, then the fixed space of $U$ on $V$ is nonzero and $\H$-invariant, and
 thus the action of $\H$ is reducible.   So we may assume instead that $\H$ is reductive.

Similarly, if the center of $\H$ is positive dimensional,  then its eigenspaces are $\H$-invariant, and thus the action of $\H$ is reducible.
So we may assume  that $\H$ is semisimple.
The dimensions of the smallest representations of a semisimple $\H$ are known  \cite{lubeck}  (and easy
to check).  In particular, if $\H$ is simple of rank $r$, then the smallest dimension of a nontrivial irreducible representation
of $\H$ is $2r$ for $\H$ of type $C$ or $D$,  $2r+1$ for type $B$ (for $r \geq 3, p \ne 2$) and greater than $2r+2$ for $\H$
exceptional.   If $\H$ is of type $A$, then the minimal dimension is $r+1$ and for $r > 1$ any such module is not self dual.   We also note
that for $r > 1$, the smallest nontrivial irreducible self dual representation of $A_r$ is of dimension greater than $2(r+1)$.

So assume that $\H$ is semisimple of rank $r$.    Write $\H=\H_1 \times \ldots \times \H_m$ where the $\H_i$ are simple
and the sum of the ranks of the $\H_i$ is $r$.
Let $W$ be an irreducible $\H$-module where the kernel is contained in the center.  Then $W$ is a tensor product
of nontrivial $\H_i$-representations.  Thus by the remarks above,  $\dim W > r +1$ unless $\H=A_r$.
Similarly,  if $W$ is self dual (equivalently each $W_i$ is self dual), it follows that $\dim W > 2r +1$ unless
possibly $\G=\Sp_4$ and $\H=\SO_4$ or $\H$ is simple of type $B_r, C_r$ or $D_r$.    Note that $\dim B_r = \dim C_r > \dim D_r$,
whence we see that the smallest representation of $B_r$ is $2r+1$ dimensional (for $p \ne 2$; if $p=2$,  $B_r \cong C_r$).
If $p \ne 2$,  $D_r$ does not embed in $C_r$ (since all representations of $D_r$ of dimension $2r$ preserve a quadratic form
while $C_r$ preserves an alternating form in any of its $2r$ dimensional representations).   This completes the proof.
\end{proof}

\begin{remark} If $\G=\Sp_4$, we may view it as $\SO_5$ instead (and then this is not a special case).  Similarly, if $p=2$
and $\G=\Sp_{2n}$ and we view this as $\SO_{2n+1}$ with its natural orthogonal module (which is indecomposable but not
irreducible), in which case $\SO_{2n}$ is the stabilizer of a nondegenerate hyperplane.
\end{remark}

In fact, in the classical case, we can give a better description of the degenerations of maximal semisimple subgroups of maximal rank as follows.

\begin{theorem}[Degenerations in classical groups]\label{classical-degeneration}
Let $\G$ be a classical group of rank $r$.   Let $\H$ be a proper semisimple maximal subgroup of $\G$ of rank $r$.
If $\H'$ is a degeneration of $\H$, then either $\H'$ is conjugate to a subgroup of $\H$ or $\H'$ is a subgroup
of a parabolic subgroup.  In particular, $\H'$ cannot be a semisimple subgroup of rank $r$ unless $\H'$ is conjugate to a subgroup of $\H$.
\end{theorem}

\begin{proof} Note that $\H$ is not contained in a parabolic subgroup (since the derived subgroup of a parabolic subgroup
has rank at most $r-1$).   By Lemma \ref{max-rank}, this implies that $\SL_{r+1}$ contains no proper semisimple subgroups of rank
$r$.   Thus, the theorem is vacuously true in this case.     So $\G$ preserves a nondegenerate quadratic or alternating
form on $V$.   For the moment,
 exclude the two special cases in   Lemma \ref{max-rank}.  Then $\H$ is the stabilizer of some $d$-dimensional subspace
$W$ of the natural module.  Then $W$ must be nondegenerate (for otherwise $H$ preserves the radical of $W$ and the stabilizer
of a totally singular subspace is a parabolic subgroup).      It follows by Corollary \ref{test}(i) that $\H'$ stabilizes some $d$-dimensional
subspace $W'$ of $V$.  If $W'$ is nondegenerate, then $\H'$ is conjugate to a subgroup of $\H$.  Otherwise, $\H'$ preserves
the radical of $W'$ and so is contained in a parabolic subgroup.

The proof in the two special cases is identical using the remark above.
\end{proof}

We can now prove Proposition \ref{degeneration-prop} in the classical case.
We leave aside the case that $\G=B_2 = C_2$ in characteristic  $3$.   Let $\G$ be a classical group and
$\H$ a semisimple maximal subgroup of $\G$ with maximal rank. The proposition is vacuous if $\G$ is of type $A_r$,
because $\G$ does not admit any semisimple subgroup of maximal rank.  Hence we may take $\G$ to be $B_r$ for $r \geq 2$, $C_r$ for $r \geq 3$, or $D_r$ for some $r \geq 3$.

If $\G$ is of type $D_r, r > 3$, then
$\H$ is of type $D_k \times D_{r-k}$ and stabilizes an even dimensional subspace of the natural module $V$.
Hence the subgroup $H':=B_{r-1} \leq D_r$ cannot be contained in a conjugate of $\H$ and by Lemma \ref{closed invariants} (i)
it cannot be a degeneration of it either (note that since $r > 3$, we are not using $B_2=C_2$ in characteristic in $3$).

We now turn to the case when $\G$ is of type $B_r$ or $C_r$. We first observe that if $p=0$ (or more generally $p > 2r$)
and  $\G$ is of type $B_r$ or $C_r$, then $\G$ contains a subgroup $\H'$ of type
$A_1$ which acts irreducibly on the natural module $V$ (recall the the $n$th dimensional irreducible representation of
$\SL_2$ preserves a non-degenerate bilinear form which is orthogonal if $n$ is odd and symplectic if $n$ is even, provided that $p > n$).
Since the maximal rank semisimple subgroups are not irreducible, we conclude from Lemma \ref{closed invariants} (i) that $\H'$ cannot be a degeneration of any $\H$.

The above argument handles the case when $r = 2$ and $p > 3$.  Now assume that $p$ is odd and $r > 2$.
  If $G=C_r$, there is a
semisimple  subgroup $SL_2 \otimes SO_r < G$ which acts irreducibly and so cannot be a degeneration
of any proper semisimple subgroup of maximal rank.   Note that $B_r$ contains two semisimple maximal subgroups --
$D_r$ and $A_1 \times D_{r-1}$.  Thus,
the result follows by Theorem \ref{classical-degeneration}.

Finally, consider the case that $p=2$ and $G=C_r,  r \ge 2$.      There are two nonconjugate maximal rank semisimple
subgroups in all cases; for instance one can take $D_r$ and  and the stabilizer of a nondegenerate $2$-space.
 Thus the result follows by Theorem \ref{classical-degeneration}.

  \vspace{3pt}

We now begin the proof of Proposition \ref{degeneration-prop} in the case when $\G$ is an exceptional group.  We split into various cases,
depending on the Dynkin diagram of $\G$ and on the characteristic $p = \ch(k) \geq 0$ of the field.
Assume that the proposition is false and let $\H$ be a maximal semisimple subgroup of maximal rank $r$ that
is a counterexample; thus every proper closed semisimple subgroup of $\G$ is a degeneration of $\H$.

As a general reference for representations of
algebraic groups, see \cite{steinberg-yale}.  Also, see \cite{lubeck} for a list of all the small dimensional representations of the simple
algebraic groups.  Finally, see \cite{lieseitz} for detailed information about subgroups of exceptional groups.
For some of the small characteristic arguments below, we use some straightforward computations to compute dimensions
of composition factors of some small modules restricted to these subgroups.

\emph{Case 1. $\G = G_2$ and $p \neq 2,3,5$.}   Up to conjugacy, there are only two semisimple subgroups of $\G$ of maximal rank,
namely $H_1 = A_1(k) \times A_1(k)$ and $H_2= A_2(k)$, and so $\H$ must be conjugate to one of these groups.
Also, $\G$ has an irreducible module $V$  of dimension $7$.   Since
$\H_1$ is the centralizer of an element of order $3$ and $\H_2$ is the centralizer of involution, neither $\H_i$ can act irreducibly (on any nontrivial
$\G$-module). On the other hand, when $p \neq 2,3,5$, $\G$ also has an $A_1$-subgroup $\H'$ that acts irreducibly on $V$, and the claim then follows from  Corollary \ref{test} (i).

\emph{Case 2. $\G = G_2$ and $p=5$.}  Here one has to proceed a little more delicately than in Case 1; in characteristic $5$ one no longer has the
irreducible $A_1(k)$-subgroup available.   In this case, $\H_1$ has composition factors of dimensions $1,3,3$ and $\H_2$ has
composition factors of dimensions $3$ and $4$ on $V$.

If $\H$ is conjugate to $\H_2$, then we are done by Lemma \ref{test} (ii) with $\H' := \H_1$ and $d := 3$, since $0 \not \geq 4$.  Thus we may
assume that $\H$ is conjugate to $\H_1$.

There is an irreducible module $U$  of dimension $27$, namely a quotient
 of the symmetric square of the $7$ dimensional module by the one dimensional trivial submodule.
 One computes that for $\H_2$, there are composition factors of dimension $6,6,3,3,8,1$.
 For $\H_1$, the composition factors have dimension $5,1, 8,4,9$.   Since
$$ 8 + 9 \not \geq 6 + 6 + 8 $$
we are then done by applying Corollary \ref{test} (ii) with $\H' := \H_2$ and $d=5$.

\emph{Case 3. $\G = G_2$ and $p=2$.}  In this case we still have that the two maximal rank semisimple subgroups up to conjugacy are
$\H_1, \H_2$, but now $G$ has an irreducible module $V$ of dimension $6$.  Then $\H_2$
has only two irreducible composition factors of  dimension $3$, while $\H_1$ has composition factors
of dimension $2$ and $4$.   Applying Corollary \ref{test} (ii) either with $\H=\H_1, \H'=\H_2, d=3$ or $\H=\H_2, \H'=\H_1, d=2$ we obtain the claim.

\emph{Case 4. $\G = G_2$ and $p=3$.}   There are three conjugacy classes of maximal rank semsimple groups in this case, two of which are isomorphic to
$A_2$ (twisted by the graph automorphism), and there are two irreducible $7$-dimensional modules.   One class of $A_2(k)$
is irreducible on one of the modules and one on the other.   The third class is  isomorphic to $A_1(k) \times A_1(k)$.
Since none of the $3$ subgroups is irreducible on both of the $7$ dimensional modules, the result follows from Corollary \ref{test} (i).

\emph{Case 5. $\G = F_4$ and $p>3$.}   Let $V$ be the irreducible module
of dimension $26$.  There is a copy of $B_4(k)$ in $\G$ that has composition factors of dimension $1,9,16$.
By Corollary \ref{test} (i) with $d=15$, we are done unless the largest composition factor of $\H$ on $V$ has
dimension at least $16$.  There is also a subgroup of type $A_1(k) \times C_3(k)$ (the centralizer of an involution)
which has two composition factors of dimensions $14$ and $12$.  Applying Corollary \ref{test} (i) with $d=11$,
 we are done unless all the composition factors of $\H$ have dimension at least $12$.  Since $16+12 > 26$,
 The only remaining case is when $\H$ just has a single composition factor of dimension $26$, i.e. $\H$ acts irreducibly on $V$.
 But no such $\H$ exists by \cite{lss}.

\emph{Case 6. $\G = F_4$ and $p=3$.}  In this case there is an irreducible module $V$ of dimension $25$, and the copy of $B_4(k)$ has two irreducible submodules of dimensions $9$ and $16$ while the copy of $A_1(k) \times C_3(k)$ has composition factors of dimension $12$ and $13$.  Repeating the Case 5 argument (noting that $16+12 > 25$), this handles all cases except the one where $\H$ acts irreducibly on $V$. However, as before, no such $\H$ exists.

\emph{Case 7. $\G = F_4$ and $p=2$.}  In this case, there are \emph{two} irreducible modules $V, V'$ of dimension $26$ (swapped via the graph automorphism, which exists only for $p=2$), and there are two conjugacy classes of subgroups isomorphic to $B_4(k)$ (interchanged by the graph automorphism).   One is irreducible on $V$ and  the other on $V'$.
Applying Corollary \ref{test} (i), we are done unless $\H$ acts irreducibly on both $V$ and $V'$, but no such $\H$ exists.

\emph{Case 8. $\G=E_6$.}  The smallest irreducible $\G(k)$-module has dimension $27$.   There is a copy of $F_4(k)$ that has an irreducible submodule of dimension $26$  ($25$ if $p=3$).  The only maximal
positive dimensional connected semisimple groups of rank $6$ are isomorphic to either $A_1(k) \times A_5(k)$ or $A_2(k) \times A_2(k) \times A_2(k)$ and neither has a composition factor of dimension at least $25$, and so we are done by  Corollary \ref{test} (ii) with $d=24$.

\emph{Case 9. $\G = E_7$.}  In this case, the smallest irreducible module $V$ for $\G(k)$ is $56$-dimensional.   There is a subgroup of type $A_1(k) \times D_6(k)$ that has composition factors of dimensions $24$ and $32$ on $V$.  There is also a
subgroup $E_6(k)$ which has composition factors of dimensions $1, 27$ (each with multiplicity $2$).  Applying Corollary \ref{test} (ii) with $\H'$ equal to the $A_1(k) \times D_6(k)$ subgroup and $d$ equal to $23$ or $31$, we are done unless $\H$ has a composition series with largest factor of dimension at least $32$, and smallest factor of dimension at least $24$.  On the other hand, applying Corollary \ref{test} (ii) with $\H' = E_6(k)$ and $d=26$, we are done unless the composition factors of $\H$ of dimension at least $27$ have total dimension at least $54$.  This rules out all cases except when $\H$ acts irreducibly on $V$.  But no such $\H$ exists.

\emph{Case 10. $\G = E_8$ and $p \neq 2$.}   Let $V$ be the adjoint module of dimension $248$.   There is a subgroup of type $D_8(k)$
which has precisely two composition factors on $V$ of dimensions $120$ and $128$.  There is also a subgroup of type $A_1(k) \times E_7(k)$ that has composition factors of dimension $3, 133$ and $112$.   By Corollary \ref{test} (ii) with $\H'$ equal to the $D_8(k)$ subgroup and $d=119,127$, we are done unless the composition factors of $\H$ have largest dimension at least $128$ and smallest dimension at least $120$, while from Lemma \ref{test} (ii) with $\H'$ equal to the $A_1(k) \times E_7(k)$ subgroup, we are done unless the largest factor has dimension at least $133$.  This covers all cases except when $\H$ acts irreducibly on $V$.  Clearly, this cannot occur since the adjoint module of $\H$ embeds in $V$.

\emph{Case 11. $\G = E_8$ and $p=2$.}  The situation here is the same as that in Case 9, except that the $D_8(k)$ subgroup now has composition factors of dimensions $1,1, 118$ and $128$, and the $A_1(k) \times E_7(k)$ subgroup has composition factors of dimensions  $1,1, 2, 132$ and $112$.
By Corollary \ref{test}, we are done unless  the largest composition factor is at least $132$ and the sum of the dimensions of the compostion factors of dimension greater
than $117$ add up to $246$.   The only possibility that is that $\H$ has a composition factor of dimension at least $246$. If $\H$ has any simple
factor other than $\SL_2$, then $\H$ will have a composition factor of dimension between $7$ and $133$ on $V$, a contradiction.
If $\H$ is the direct product of $8$ copies of $\SL_2$, then any irreducible module has dimension a power of $2$, also a contradiction.

This exhausts all the possible cases for the simple group $\G$, and the proof of Proposition \ref{degeneration-prop} is complete.

\section{Non algebraically closed fields, characteristic 0 and finite fields}\label{finite-dense}

The material in this section is devoted to some refinements of Theorem \ref{sec3-key-prop} when the field of definition is changed, for instance when it is no longer algebraically closed, has characteristic zero or is a finite field. In particular, we prove Corollary \ref{finite}.



In characteristic zero, the assumption on the transcendence degree of $k$ in Lemma \ref{singlewords} is irrelevant and
 moreover the set of good pairs is open:

\begin{theorem}[Characteristic $0$ case]\label{char 0}
Suppose that $\G(k)$ is a semisimple algebraic group over a field $k$ of characteristic zero, and that
$w,w' \in F_2$ are noncommuting words. Then
\[ X:=\{(a,b) \in \G(k) \,  | \, \overline{\langle  w(a,b), w'(a,b)\rangle}  = \G\}\] is
an open subvariety of  $\G \times \G$ defined over $k$ and $X(k)$ is non-empty.
\end{theorem}

\begin{proof} Let $\overline{k}$ be the algebraic closure of $k$. By\footnote{This claim also follows from the combination of Lemma \ref{gurtiep} and Lemma \ref{closed invariants} (i).} \cite[Theorem 3.3]{gurnewton}, the set of pairs in a simple algebraic group over an algebraically closed field of characteristic $0$
that generate a dense subgroup is open. It follows that the same is true for $X$ (which by definition is defined over $k$).  The only issue is that it is possible that $X(k)$ might be empty.

Let $X'$ denote the (closed) complement of $X$. If $k'$ is an uncountable field extension of $k$, then it follows by
 Theorem \ref{sec3-key-prop} that $X'(k')$ is not all of $\G(k') \times \G(k')$.   Thus, $\dim X' < 2 \dim \G$.  Since $\G(k) \times \G(k)$ is dense in $\G \times \G$ (see \cite{borel-alggroups}), it cannot be entirely contained in $X'$. Thus $X(k)$ is non-empty.
\end{proof}

We remark that in positive characteristic Theorem \ref{char 0} is no longer true.  If $k$ is algebraic over a finite field, then $\G(k)$ is locally finite.
More generally, if $k'$ is algebraically closed and $k$ is the subfield of $k'$ that consists of algebraic elements over the
prime field, then $\G(k)$ is dense in $\G(k')$ and so the set of pairs $(a,b)$ such that $\langle w_1(a,b), w_2(a,b) \rangle$
is finite is dense.   We first need to recall a result that is essentially \cite[Theorem 11.6]{gt}.

Let $k$ be an algebraically closed field of characteristic $p$ (possibly $p=0$).
Let $\G$  be a simply connected simple algebraic group over $k$. We may assume that $\G$ is defined and split over a prime field (i.e. $\F_p$ or $\Q$ if $p=0$).

\begin{lemma} \label{gurtiep}
There exists a finite collection $V_1,\ldots,V_t$
 of irreducible finite dimensional $\G(k)$-modules with the property that if a proper closed subgroup
 of $\G(k)$ acts irreducibly on $V_i$ for each $1 \leq i \leq t$, then it is finite and conjugate to $\G(\F_q)$ for some finite field $\F_q$ with $p|q$.  In particular, if $\ch(k)=0$, there are no such subgroups.
 \end{lemma}

 \begin{proof}   If $\G(k)$ is classical, this is precisely \cite[Theorem 11.6]{gt}.
 If $\G(k)$ is exceptional, it follows easily by the main results of
 \cite{ls-adjoint}.
 \end{proof}

\noindent \emph{Remark.}   Note that for any fixed irreducible $\G(k)$-module with $k$ of positive characteristic,
$\G(\F_q)$ will be irreducible for $q$ sufficiently large.  If the module is restricted, this is true without exception.
In almost all cases, the collection of $V_i$ given above can be taken to be restricted. We will only use the
fact that the subfield groups are the only possible closed subgroups that are irreducible on the collection
of modules.
We note further that for fixed rank and $p$ sufficiently large,  the $V_i$ are independent of $p$  (indeed, aside
from very small characteristics, the dominant weights can be chosen independent of characteristic).

Let us comment a bit on the previous lemma. When $\ch(k)=0$, we can take $V_1$ to be the Lie algebra of $\G$ under the adjoint representation of $\G$. Then, if a closed subgroup $\H$ of $\G$ acts irreducibly on $V_1$, it must be finite, because otherwise $\H$ stabilizes its Lie subalgebra, which is a non trivial proper subspace of $V_1$. Now by Jordan's theorem, there is a bound $m(d)>0$ ($d=\dim(\G)$) such that every finite subgroup of $\G(k)$ has an abelian normal subgroup of index at most $m(d)$. In particular a finite subgroup of $\G(k)$ cannot act irreducibly on any vector space of dimension $>m(d)$. Then, simply set $V_2$ to be any irreducible $\G$-module of dimension $>m(d)$. Thus we see that $t=2$ is enough in the above lemma when $\ch(k)=0$. With a bit more effort, one can show that there is a single irreducible module that suffices.

When $\ch(k)>0$, the idea is similar, but the details are more involved. To begin with, the adjoint representation may not be irreducible. For example, the Lie algebra $\mathfrak{sl}_{dp}$ of $\SL_{dp}$ ($d \in \N$, $p$ prime) contains the scalar matrices as a one-dimensional submodule. However, if $p=\ch(k)>3$, then $\SL_{dp}$ is the only exception and, even for $\SL_{dp}$, the adjoint representation has only $2$ composition factors (of dimensions $1$ and $n^2-2$). So when $\ch(k)>3$, we can still set $V_1$ to be the adjoint representation (or its non trivial composition factor in case of $\SL_{dp}$) and it will still be true that every proper closed subgroup $\H$ of $\G$ that acts irreducibly on $V_1$ must be finite. At this point, as in \cite[Theorem 11.7]{gt}, one can invoke the main result of Larsen-Pink \cite{larsen-pink} according to which either $\H$ is a subfield subgroup, or $|\H|\leq m(d)$ for some constant $m(d)>0$. Then, as above, taking for $V_2$ any irreducible $\G$-module of dimension $>m(d)$, we see that if a closed subgroup of $\G$ is irreducible on both $V_1$ and $V_2$, then it is a subfield subgroup. The characteristics $2$ and $3$ require a more careful case by case analysis in the same spirit. See \cite[Theorem 11.17]{gt} for more details. \vspace{11pt}

We can now prove the following result.

\begin{corollary} \label{dense char p}    Suppose that $\G(k)$ is a semisimple algebraic group over an algebraically closed field $k$ of
characteristic $p > 0$ with $k$ not algebraic over a finite field.    If $p=3$, assume that $\G(k)$ has no factors of type
$C_2$.  Let $w_1, w_2 \in F_2$ be non-commuting words.
Then for a dense set of $(a,b) \in \G(k) \times \G(k)$,
one has $\overline{\langle w_1(a,b), w_2(a,b) \rangle} = \G(k)$.
\end{corollary}

\begin{proof}  We may assume that $\G(k)$ is simple.
Let $V_1, \ldots, V_m$ be as in Lemma \ref{gurtiep} a collection of finite dimensional irreducible $\G(k)$ modules such that every
closed subgroup of $\G(k)$ that is irreducible on all the $V_i$ is a subfield subgroup (i.e. is a conjugate of  $\G(\F_q)$ for some $q$). Clearly, the set of $(a,b) \in \G(k) \times \G(k)$ such that
$\langle w_1(a,b), w_2(a,b) \rangle$ acts irreducibly on each of the $V_i$ is open and defined over $\F_p$ (see Lemma \ref{closed invariants} (i)).   By Theorem \ref{sec3-key-prop} ,  this set is non-empty
over an uncountable field, and therefore, as in
the proof of Theorem \ref{char 0}, also over any algebraically closed field
(including the algebraic closure of a finite field).   If  $k$ is not algebraic over a finite
field, it is easy to see that the set of elements of $\G(k)$ that have infinite
order is dense (this reduces to the case of a
$1$-dimensional torus).  However, the preimage of a dense subset by dominant map is itself dense. So by Borel's theorem on the dominance of the word map (i.e. Proposition \ref{borel-words-prop}), we get that the set of pairs $(a,b)$ such that $w_1(a,b)$ is of infinite order and such that $\langle w_1(a,b), w_2(a,b) \rangle$ is irreducible on each of the $V_i$'s is  dense in $\G(k) \times \G(k)$, whence the result.
\end{proof}

If $k$ is algebraic over a finite field,  $\G(k)$ is locally finite and the statement of Corollary \ref{dense char p} needs to be altered. Generically, $w_1(a,b)$ and $w_2(a,b)$ will not generate a dense subgroup, but they will generate a \emph{subfield} subgroup, that is a conjugate of $\G(\F)$ for some finite field $\F$. We have the following.

\begin{proposition}\label{locally-finite} Let $\F_q$ be a finite field and $k$ its algebraic closure. Suppose that $\G$ is a simply connected
simple algebraic group over $\F_q$ with $\G$ not $C_2$ if $q$ is a power of $3$.
   Let $X$ be the non-empty open subvariety of $\G(k) \times \G(k)$  consisting of those pairs $(a,b)$
which generate an irreducible subgroup of $G$ on each of the modules $V_i$ defined in Lemma \ref{gurtiep}.
 Let $w_1, w_2 \in F_2$ be non-commuting words. Then
$$  X \subseteq \{(a,b) \in \G(k) \times \G(k):  \langle w_1(a,b), w_2(a,b) \rangle  = \G(\F_{q_0})^g \text{ for some } q_0, g\}. $$
\end{proposition}

\begin{proof} As in the proof of Corollary \ref{dense char p}, this follows immediately from Lemma \ref{gurtiep} together with Lemma \ref{closed invariants} (i).
\end{proof}

\noindent \emph{Proof of Corollary \ref{finite}}.
Let $r$, $G=\G(\F_q)$, and $w_1, w_2$ be as in Corollary \ref{finite}, and
let $k$ be the algebraic closure of $\F_q$.
Aside from small characteristics, the $V_i$ used to describe the open variety $X$ from Proposition \ref{locally-finite} are independent of the characteristic. It follows then from the Lang-Weil bound
(\cite[Prop. 3.4]{larsen-shalev-pre})
that with probability $1 - o_{q \to \infty}(1)$, a random $(a,b) \in \G(\F_q)$ will be such that
$\langle w_1(a,b), w_2(a,b) \rangle = \G(\F_{q_0})^g$ for some $g \in \G(\F_q)$ and some $q_0$
with $\G(\F_{q_0}) \leq \G(\F_{q})$.

By \cite[Theorem 1.1]{fulmangur},  the number of conjugacy classes in $G(\F_{q_0})$ is $O(q_0^r)$,
where the implied constant is absolute. Also, by \cite[Theorem 1.4]{fulmangur}, each conjugacy class in
$\G(\F_q)$ has size
\[ \ll|\G(\F_q)|(1 + \log_q r)|/q^r,\] where the implied constant is again absolute.
Thus, \[\left|\bigcup_{q_0 < q, g \in G} \G(\F_{q_0})^g \right| \ll (1 + \log_q r) \log_2 q |\G(\F_q)|/q^r .\]
It now follows by \cite[Lemma 2.1]{larsen-shalev-jams} or \cite[Prop. 3.4]{larsen-shalev-pre}
that the probability for random $(a,b)$ that $w_1(a,b)$ (or $w_2(a,b)$) belongs to $\G(\F_{q_0})^g$ for some proper divisor $q_0$ of  $q$
tends to $0$ as $q \rightarrow \infty$ (for the groups aside from the Suzuki and Ree groups, this is essentially
the Lang-Weil theorem \cite{lang-weil} combined with Borel's theorem on word maps -- in the other cases, the references above involve a twisted version of Lang-Weil).
Combining these two remarks gives Corollary \ref{finite}.
\endproof

We end this section by showing that Theorem \ref{sec3-key-prop} continues to hold with weaker hypotheses on the field $k$.

\begin{theorem}\label{rationality issues}
Let $\G$ be a semisimple algebraic group defined over a field $k$ and $k'$ a field extension of $k$.
If $\ch(k)=3$, assume that $G$ has no factors of type $C_2$. In either of the following two situations, $\G(k')$ contains a strongly dense free subgroup:
\begin{itemize}
\item[(i)] $k'$ is algebraically closed and its transcendence degree over $k$ is at least $2\dim \G$.
\item[(ii)] $k'$ has infinite transcendence degree over its prime field.
\end{itemize}
\end{theorem}

\begin{proof}
To prove item (i), observe that, according to Theorem \ref{sec3-key-prop}, if $K$ is an uncountable algebraically closed field extension of $k$, then $\G(K)$ contains a pair $(a,b)$ of group elements that generate a strongly dense free subgroup. If $K'$ is the subfield of $K$ generated by $k$ and the coefficients of $a$ and $b$, then its transcendence degree over $k$ is at most $2\dim\G$. In particular, $K'$ can be embedded in $k'$ and we are done.

To see (ii), we recall \cite[Lemma 2.2]{borel-dom} according to which every countable union of proper closed subvarieties of $\G \times \G$ all defined over the algebraic closure of a given finitely generated subfield $l$ of $k$ cannot cover all of $\G(k) \times \G(k)$. This is the case in particular for the family of varieties $Y_{i,w_1,w_2}$ defined (for some non commuting pair of words $w_1,w_2$ in the free group) as the subset of pairs $(a,b)$ such that $\langle w_1(a,b), w_2(a,b) \rangle$ does not act irreducibly on the $\G$-module $V_i$ defined in Lemma \ref{gurtiep}. Each such variety is defined over the algebraic closure of any finitely generated subfield of $k$ on which $\G$ is defined. Also each $Y_{i,w_1,w_2}$ is closed by Lemma \ref{closed invariants} $(i)$ and proper by Theorem \ref{sec3-key-prop}. Finally, observe that any pair $(a,b)$ that lies in $\G(k) \times \G(k)$ and outside all the $Y_{i,w_1,w_2}$ must generate a strongly dense free subgroup. Indeed, since $\langle w_1(a,b),w_2(a,b)\rangle$ acts irreducibly on each $V_i$, we must have $w_1(a,b)\neq 1$ for all non-trivial words $w_1$; this means that $(a,b)$ generates a free subgroup and that each $\langle w_1(a,b),w_2(a,b)\rangle$ is infinite and therefore dense by Lemma \ref{gurtiep}.
\end{proof}

\noindent \emph{Remark.} It is plausible that $\G(k)$ always contains a strongly dense free subgroup as soon as $k$ is not algebraic over a finite field. See Problem 1 in Section \ref{sec1}.


\section{On Borel's theorem and the Hausdorff-Banach-Tarski paradox}\label{secparadox}

In this section, we explain the consequences of our main theorem for the simultaneous paradoxical decompositions of homogeneous spaces. Borel's paper \cite{borel-dom}, where the dominance of word maps was established, was in part motivated by a question of Dekker \cite{dekker1, dekker2} and a related paper by Deligne and Sullivan \cite{deligne-sullivan} regarding instances of the Hausdorff-Banach-Tarski paradox on higher dimensional spheres. In particular, in combination with the above cited work of Dekker, the papers by Deligne-Sullivan and by Borel established that every $n$-dimensional sphere ($n \geq 2$) is \emph{$4$-paradoxical}; namely that one can find two rotations $a,b \in \SO(n+1)$ and partition the sphere $S^n$ into $4$ disjoint parts $A_1 \cup A_2 \cup A_3 \cup A_4$ such that both $A_1 \cup aA_2$ and $A_3 \cup bA_4$ are partitions of the same sphere $S^n$.

It is well-known (see \cite[Chapter 4]{wagon}) that a sufficient condition for a $G$-set $X$ to be $4$-paradoxical is the existence in $G$ of a free subgroup on two generators whose action on $X$ is \emph{locally commutative}. By definition, this means that point stabilizers are commutative. In particular, every free action (i.e. where every nontrivial element acts without fixed points) is locally commutative. Thus Deligne and Sullivan establish that there is a free action of a free group of rotations on every odd dimensional sphere of dimension at least $3$. For even dimensional spheres $S^n$, no free action is possible of course, because every rotation has an axis and thus fixed points. Nonetheless, Dekker showed that if $n \neq 4$ is even, then a locally commutative action of a free group of rotations can be found, and later, in that same paper, Borel gave a different argument valid for all $n\geq 2$ including $n=4$.

In fact, Borel extended the Deligne-Sullivan result in two directions. First, using his result on the dominance of the word map, he proved (see  \cite[Theorem A]{borel-dom}) that generic pairs of elements in a compact semisimple Lie group $U$ generate a free subgroup $F$ with the property that every nontrivial element of $F$ acts on each homogeneous space $U/V$ with the minimal possible number of fixed points, that is the Euler characteristic $\chi(U/V)$ as prescribed by the Lefschetz fixed point formula. Second, by a different argument which did not rely on the dominance of the word map result,  he showed that if the subgroup $V$ is given and of maximal rank in $U$, then one may find a free subgroup $F$ in $U$ whose action on $U/V$ is locally commutative (see \cite[Theorem 3]{borel-dom}). This last statement answered Dekker's question and showed that the $4$-dimensional sphere is $4$-paradoxical.

Using our Theorem \ref{sec3-key-prop}, we can unify both theorems and thus extend their conclusion simultaneously to all homogeneous spaces of any given semisimple compact Lie group, as follows.

\begin{theorem}\label{U/V} Let $U$ be a semisimple compact real Lie group. Almost every pair $a,b \in U$ \textup{(}both in the sense of Baire and in the measure theoretic sense\textup{)} has the following property. The subgroup $F=\langle a , b \rangle$ is free, and for every proper closed\footnote{Note that every subgroup of $U$ which is closed for the Euclidean topology is also closed for the Zariski topology (see e.g. \cite{vinberg}).} subgroup $V$ of $U$, the action of $F$ on $U/V$ is locally commutative \textup{(}i.e. point stabilizers are commutative\textup{)}, and every nontrivial element of $F$ generates topologically a maximal torus and has precisely the minimum possible number of fixed points on $U/V$, namely the Euler characteristic $\chi(U/V)$.
\end{theorem}

Classically, this has the following consequence in terms of paradoxical decompositions.

\begin{corollary} Let $U$ be a semisimple compact Lie group. Then there exists $a,b \in U$ such that every homogeneous space of $U$ is $4$-paradoxical with respect to $a$ and $b$. In other words, for every proper closed subgroup $V$ in $U$, one can partition $U/V$ into $4$ disjoint sets $A_1 \cup A_2 \cup A_3 \cup A_4$ in such a way that $A_1 \cup aA_2$ and $A_3 \cup bA_4$ are both partitions of $U/V$. In fact almost every pair $a,b$ in $U$ has these properties.
\end{corollary}

Note that we recover the sphere $S^n$ as the homogeneous space $U/V$, where $U=\SO(n+1,\R)$ and $V=\SO(n,\R)$. Other examples of such $U/V$ include the complex and quaternionic projective spaces and all symmetric spaces of compact type. We refer the reader to Wagon's book \cite{wagon} for a thorough exposition of the Hausdorff-Banach-Tarski and related problems.\vspace{5pt}

\noindent \emph{Proof of Theorem \ref{U/V}.} The compact Lie group $U$ can be written $U=\U(\R)$, where $\U$ is a semisimple algebraic group defined over $\R$. Moreover every proper (Euclidean)-closed subgroup $V$ of $U$ is algebraic and of the form $V=\V(\R)$ for some algebraic subgroup $\V$ defined over $\R$ (see e.g. \cite{vinberg}). By Theorem \ref{sec3-key-prop} (more precisely Theorem \ref{generic strongly dense}), almost every pair $(a,b)$ in $U$ generates a strongly dense free subgroup $F$ of $U$ (note that the notion of Zariski genericity used in Theorem \ref{generic strongly dense} is stronger than genericity in the sense of Baire in $U\times U$ and in the sense of Lebesgue measure). In particular, given any proper closed subgroup $V$, $V \cap F$ is trivial or infinite cyclic. This implies that the action of $F$ on $U/V$ is locally commutative in the above sense. Moreover (see Lemma \ref{generic-max-rank}), every element $\gamma$ of $F$ generates a subgroup whose closure (in both the Zariski and Euclidean topologies) is a maximal torus $T_\gamma$ of $U$. So the set of fixed points of $\gamma$ on $U/V$ is the set of fixed points of $T_\gamma$. It is $0$ unless $V$ contains a maximal torus $T$, in which case it is equal to the index of $N_V(T)$ in $N_U(T)$. In all cases this number coincides with $\chi(U/V)$ (see \cite{borel-dom} and the reference \cite{hopf} cited there).
\endproof

\noindent \emph{Remark.} As mentioned above, Borel proved a weaker form of Theorem \ref{U/V} in which the choice of $F$ was not independent of $V$ and $V$ was assumed to have maximal rank. His proof consisted first in finding a special copy of $\SO(3,\R)$ inside $U$ that was not contained in any proper subgroup of maximal rank and then pick $F$ inside this copy of $\SO(3,\R)$. Since every free subgroup of $\SO(3,\R)$ is strongly dense, the argument was complete. Of course, this method cannot be applied to prove our strengthening (i.e. Theorem \ref{U/V}), because such an $F$ lies in a proper closed subgroup and one may then choose $V$ to be precisely that subgroup, in which case $U/V$ is not paradoxial with respect to $F$, because $F$ has a global fixed point.\vspace{10pt}

\noindent \emph{Remark.} The above results extend without difficulty to non-compact semisimple real Lie groups $U$ provided the homogeneous space $U/V$ is such that $V$ is a proper \emph{algebraic} subgroup of $U$.

\setcounter{tocdepth}{1}

\end{document}